\documentclass[preprint,sort&compress,12pt]{elsarticle}
\usepackage{amscd,amssymb,amsmath,amsthm}
\usepackage[all]{xy}
\usepackage{array}
\usepackage{etoolbox}
\usepackage{mathtools}
\usepackage[most]{tcolorbox}
\usepackage{tikz-cd}

\makeatletter
\def\ps@pprintTitle{%
	\let\@oddhead\@empty
	\let\@evenhead\@empty
	\def\@oddfoot{\reset@font\hfil\thepage\hfil}
	\let\@evenfoot\@oddfoot
	
}

\newdir{ >}{!/8pt/\dir{}*\dir{>}}

\newtheorem{theorem}{Theorem}[section]

\newtheorem{lemma}[theorem]{Lemma}
\newtheorem{corollary}[theorem]{Corollary}
\newtheorem{prop}[theorem]{Proposition}

\newtheorem*{con7*}{Conjecture 7*}
\newtheorem{Remark}{Remark}
\theoremstyle{definition}
\newtheorem{definition}[theorem]{Definition}

\begin{document}

\begin{frontmatter}

\title{Invariance of Schur multiplier, Bogomolov multiplier and the minimal number of generators under a variant of isoclinism}

 \author[IISER TVM]{A. E. Antony}
\ead{ammueliza@gmail.com}
 \author[IISER TVM]{Sathasivam K}
\ead{sathasivam19@iisertvm.ac.in}
\author[IISER TVM]{V.Z. Thomas\corref{cor1}}
\address[IISER TVM]{School of Mathematics,  Indian Institute of Science Education and Research Thiruvananthapuram,\\695551
Kerala, India.}
\ead{vthomas@iisertvm.ac.in}
\cortext[cor1]{Corresponding author. \emph{Phone number}: +91 8921458330}

\begin{abstract}
We introduce the $q$-Bogomolov multiplier as a generalization of the Bogomolov multiplier, and we prove that it is invariant under $q$-isoclinism. We prove that the $q$-Schur Multiplier is invariant under $q$- exterior isoclinism, and as an easy consequence we prove that the Schur multiplier is invariant under exterior isoclinism. We also prove that if $G$, $H$ are $p$-groups and $G/Z^{\wedge}(G)\cong H/Z^{\wedge}(H)$, then the cardinality of the minimal number of generators of $G$ and $H$ are the same. Moreover we prove some structural results about $q$-nonabelian tensor square of groups. 
\end{abstract}


\end{frontmatter}

 \section{Introduction}
 
The notion of isoclinism was introduced by P. Hall in \cite{PH1940}. It is a well-known fact that the Schur Multiplier of any group and the minimal number of generators of a $p$-group are not invariant under isoclinism. In this paper, we define a variant of isoclinism and prove that the Schur Multiplier and the bogomolov multiplier are invariant under this modified version of isoclinism. Invariance of the Bogomolov Multiplier under isoclinism was proved by Bogomolov and B\"{o}hning in \cite{BB} and by Moravec in \cite{PM10}. We define two different notions of $q$-isoclinism and consider two variants of the Bogomolov Multiplier, and we further prove their invariance under $q$-isoclinism. We also prove that if $G$, $H$ are $p$-groups and $G/Z^{\wedge}(G)\cong H/Z^{\wedge}(H)$, then $d(G)=d(H)$, where $d(G)$ is the minimal number of generators of $G$. For a group $G$ and a positive integer $q$, the nonabelian $q$-tensor product and nonabelian $q$-exterior product for $G$-crossed modules were introduced in \cite{CF}, as a generalization of definitions given in \cite{B} and \cite{ER}. In this paper, we consider a special case of this where the crossed modules are given by a group $G$ and the identity morphism $Id: G\rightarrow G$, called the nonabelian q-tensor square of $G$. In \cite{ADST}, the authors prove several structural results for the case $q = 0$. We generalize these structural results for a general $q$ and generalize some result of \cite{RR2017}.

\section{Preliminary Results}

 For $q\geq 1$, the $q$-nonabelian tensor square is defined as follows.
 \begin{definition}
 The \textit{tensor square modulo q}, $G\otimes^q G$ of the $G$-crossed module $Id: G \rightarrow G$ is the group generated by the symbols $g\otimes h$ and $\{( g,g )\}, g, h \in G$  with the following relations for $g,h,g_1,h_1 \in G:$
 \begin{equation}\label{eq:E:2.11}
gg_1 \otimes h = (^gg_1\otimes\ ^gh)(g\otimes h),
 \end{equation}
 \begin{equation}\label{eq:E:2.12}
   g\otimes hh_1 = (g\otimes h)(^hg \otimes\ ^hh_1),
 \end{equation}
 \begin{equation}\label{eq:E:2.13}
 \{(g,g)\}(g_1\otimes h_1)\{(g,g)\}^{-1} =\hspace{2mm} ^{g^{q}}g_1\otimes\hspace{0.5mm} ^{g{^q}}h_1,
 \end{equation}
 \begin{equation}\label{eq:E:2.14}
 \{(gg_1,gg_1)\} = \{(g,g)\}\prod_{i=1}^{q-1}(g^{-1} \otimes (^{g^{1-q+i}}g_1)^i)\{(g_1,g_1)\},
 \end{equation}
 \begin{equation}\label{eq:E:2.15}
 [\{(g,g)\},\{(g_1,g_1)\}] = g^q \otimes {g_1}^q,
 \end{equation}
 \begin{equation}\label{eq:E:2.16}
 \{([g,h],[g,h])\} = (g\otimes h)^{q}.
 \end{equation}
 \end{definition}
 For $q= 0$, the $q$-tensor square is defined to be the construction introduced by R. Brown and J.-L. Loday in \cite{BL1} and \cite{BL2}, called the nonabelian tensor square of the group $G$. Let $\nabla^q(G)$ denote the normal subgroup in $G\otimes ^q G$ generated by elements of the form $g \otimes ^q g$, where $g\in G$. 
 The \textit{exterior square modulo q}, $G\wedge^q G$, is the quotient of the group $G\otimes ^qG$ by $\nabla^q(G)$. The coset represented by the element $g\otimes^q h$ is denoted by $g\wedge ^qh$. Let $\Delta^q(G)$ be the normal subgroup in $G\otimes^q G$ generated by elements of the form $(g \otimes^q h)(h\otimes^q g)$, where $g,h \in G$. 
 
 Now we collect some results that we will need in later sections. 
 
\begin{lemma}\label{L:1.2.3}
Let $G$ be a group. For $g,g_1,h,h_1 \in G$ and $n \in Z$, the following hold in $G\otimes ^q G$.
\begin{itemize}
\item[(i)] $\nabla^q(G)$ and $\Delta^q(G)$ commute with image of $G\otimes G$ in $G\otimes ^q G$ under the natural map. 
\begin{align*}
 [g\otimes g, a \otimes b] = 1,\\
 [(g \otimes h)(h \otimes g), a \otimes b] = 1,
\end{align*} for all $g,h,a,b \in G$. In particular, both  $\nabla^q(G)$ and $\Delta^q(G)$ are abelian.
\item[(ii)]  $\exp(\nabla^q(G)) \mid q$ for $q > 0$.
\item[(iii)] $G$ acts trivially on $\nabla^q(G)$. In particular, $G$ acts trivially on $\Delta^q(G)$.
\item[(iv)]$ [g\otimes h, g_1 \otimes h_1] =  [g,h] \otimes [g_1,h_1].$

\item[(v)] 
$((g^{-1}\otimes g_1)(g_1 \otimes g^{-1}))^{-1} = (g_1 \otimes g)(g \otimes g_1).$
\item[(vi)]$(gg_1 \otimes gg_1) = (g \otimes g)(g_1\otimes g)(g\otimes g_1)(g_1 \otimes g_1).$
In particular, $\Delta^q(G) \subseteq \nabla^q(G)$.
\item[(vii)] If $[g,h] =1$, then
$(g\otimes h^n) = (g \otimes h)^n = (g^n \otimes h).$
\item[(viii)] $(g \otimes h)(g_1 \otimes h_1)(g \otimes h)^{-1} =\ ^{[g,h]}(g_1 \otimes \ h_1).$ 

\item[(ix)]  If $q$ is odd and $[g,h] = 1$, then
$\{(gh,gh)\} = \{(g,g)\}\{(h,h)\}.$
 \item[(x)]
$(g \otimes h^n)(h^n \otimes g) = \big((g \otimes h)(h \otimes g)\big)^n= (g\otimes h)^n(h\otimes g)^n$
 \item[ (xi) ] If $x \in [G,G]$, then $x \otimes x = 1$ and $(x \otimes g)(g \otimes x) = 1$ for every $g \in G$.
 \item[(xii)] $g^{-1}\otimes g = (g\otimes g)^{-1}=g\otimes g^{-1}$ and $g^{-1}\otimes g^{-1} = g\otimes g.$
  \item[(xiii)]$\{(1,1)\}=1.$

\end{itemize}
\end{lemma}

\section{Invariance of Schur Multiplier and the minimal number of generators}
 
The notion of isoclinism was introduced by P. Hall in \cite{H}. Two groups $G$ and $H$ are said to be isoclinic if there exists isomorphisms $\alpha: G/Z(G) \rightarrow H/Z(H)$ and $\beta: G' \rightarrow H'$ such that the following diagram commutes:

{\begin{equation*}
\xymatrix@+20pt{
G/Z(G)\times G/Z(G)\ar@{->}[r]^{\alpha\times \alpha}
\ar@{->}[d]^{\pi_1}
 &H/Z(H)\times H/Z(H)
\ar@{->}[d]^{\pi_2}
 \\
G'\ar@{->}[r]^{\beta}
 &H'
},
\end{equation*}} where $\pi_1$ and $\pi_2$ are the commutator maps.

It is well-known that the Schur multiplier and the minimal number of generators of a $p$-group is not invariant under the notion of isoclinism. In this section, we define a variant of the notion of isoclinism under which the Schur multiplier and the minimal number of generators of a $p$-group become invariant. 

Consider the homomorphism $\eta : G \wedge ^q G \rightarrow G$ defined by
$x \wedge y \mapsto [x,y]$ and $\{(g,g)\} \mapsto g^q$ in \cite{CF}. The kernel of this homomorphism is called the $q$-Schur Multiplier of $G$, and is denoted by $M^q(G)$. For $q\geq 1$, we have the following normal subgroups of $G$:
\begin{align*}
    Z^{\wedge}(G) &= \{g\in G\mid g\wedge h =1, \ \forall h\in G\},\\
    Z^{\wedge}_q(G) &= \{g\in G\mid g\wedge^q h =1, \ \forall h\in G\},\\
    E^{\wedge}_q(G) &= \{g\in Z^{\wedge}_q(G)\mid \{(g,g)\}=1\in G\wedge^q G\}.
\end{align*}For $q=0$, set $E^{\wedge}_0(G) := Z^{\wedge}(G)$.

\begin{definition}\label{def:5.1}
We define two groups $G$ and $H$ to be $q$-exterior isoclinic if there exists a pair of isomorphisms $\alpha : G/E_q^{\wedge}(G) \rightarrow H/E_q^{\wedge}(H)$ and $\beta : G^q[G,G] \rightarrow H^q[H,H]$ with the property that whenever $\alpha(g_1E_q^{\wedge}(G)) = h_1E_q^{\wedge}(H)$ and $\alpha(g_2E_q^{\wedge}(G)) = h_2E_q^{\wedge}(H),$ then $\beta([g_1,g_2]) = [h_1,h_2]$ and $\beta(g_1 ^q) = h_1 ^q$ for $g_1,g_2 \in G$. When $q=0$, we will call $0$-exterior isoclinism as exterior isoclinism.
\end{definition}

The exterior isoclinism is an equivalence relation and all cyclic groups are exterior isoclinic to \{1\}.

\begin{definition}\label{def:5.2}
We define two groups $G$ and $H$ to be weak $q$-exterior isoclinic if there exists subgroups $A \leqslant E^{\wedge}_q(G)$, $B \leqslant E^{\wedge}_q(G)$ and a pair of isomorphism $\alpha : G/A \rightarrow H/B$, $\beta : G^q[G,G] \rightarrow H^q[H,H]$ with the property that whenever $\alpha(g_1A) = h_1B$ and $\alpha(g_2A) = h_2B,$ then $\beta([g_1,g_2]) = [h_1,h_2]$ and $\beta(g_1 ^q) = h_1 ^q$ for $g_1,g_2 \in G$
\end{definition}

\begin{theorem}[Invariance of $q$-Schur multiplier]\label{T:5.1}
 Let $G$ and $H$ be two groups, and $A \leqslant E^{\wedge}_q(G)$, $B \leqslant E^{\wedge}_q(H)$.
\begin{enumerate}
\item[(i)]  If $G/A {\cong} H/B$, then $G\wedge^q G \cong H\wedge^q H.$
\item[(ii)] If $G$ and $H$ are weak $q$-exterior isoclinic then $M^q(G)\cong M^q(H)$. In particular, $M^q(G)$ is invariant under $q$-exterior isoclinism.
\end{enumerate}
\end{theorem}

\begin{proof}
   \textit{(i)} Consider the exact sequence $A \wedge^q G \rightarrow G \wedge^q G \rightarrow G/A \wedge^q G/A \rightarrow 1$ found in \cite[Proposition 7]{GE4}. Noting that the image of $A \wedge^q G$ is trivial in $G \wedge^q  G$ yields $G \wedge^q G \cong G/A \wedge^q G/A$. The isomorphism $\phi: G/A {\cong} H/B$ induces an isomorphism $\bar{\phi}: G/A\wedge^q G/A \to H/B \wedge^q H/B.$ Hence we obtain, $G\wedge^q G \cong  G/A\wedge^q G/A \cong H/B \wedge^q H/B \cong H\wedge^q H.$ 
   
 \par \textit{(ii)} Consider the following commutative diagram : 

\[
 \begin{tikzcd}
1\arrow{r}& M^q(G)\arrow{r}\arrow{d}& G\wedge^q G \arrow{r}\arrow{d}{\bar{\phi}} & G^q{[G,G]} \arrow{r}\arrow{d}{\beta} &1\\
1\arrow{r}& M^q(H)\arrow{r} &H\wedge^q H \arrow{r} & H^q{[H,H]}\arrow{r} &1.
\end{tikzcd}\\
\]
Using \textit{(i)}, $\bar{\phi}$ is isomorphism. Now the short five lemma yields the result.
\end{proof}

\begin{corollary}[Invariance of Schur Multiplier under exterior isoclinism]
If $G$ and $H$ are two groups that are exterior isoclinic, then $M(G)\cong M(H)$.
\end{corollary}

\begin{corollary}\label{C:5.1}
Let $G$ be a group, and $H\leqslant G$, $K\leqslant Z^{\wedge}(G).$ If $G=HK$ and $H\cap K\leqslant Z^{\wedge}(H)$, then $M(G)\cong M(H)$.
\end{corollary}
\begin{proof}
$G/K= HK/K \cong H/(H\cap K)$ and the inclusion $i:[H,H]\to[G,G]$ is an isomorphism. It is easy to verify $G$ and $H$ are weak exterior isoclinic and hence the result follows from theorem \ref{T:5.1} \textit{(ii)}.
\end{proof}
\begin{definition}
We say $G$ is internal central product of $H$ and $K$ amalgamating $D$ if $G=HK$, $H\cap K =D$ and $[H,K]=1$
\end{definition}

\begin{corollary}\label{C:5.1}
Let $G$ be a internal central product of $H$ and $K$ amalgamating $D$. If $K$ cyclic, $( \vert H/D\vert,\vert K/D\vert)=1$ and $D\leqslant Z^{\wedge}(H)$, then $M(G)\cong M(H).$
\end{corollary}

\begin{proof}
   It suffices to show $K\leqslant Z^{\wedge}(G)$. Consider the morphisms $H\wedge H\to G\wedge G$. Since $D\leqslant Z^{\wedge}(H)$, we have $h\wedge d=1$ in $H\wedge H$ and hence $h\wedge d=1$ in $G\wedge G$, for $h\in H$, $d\in D$. Similarly $k\wedge k_1 =1 $ in $G\wedge G$, where $k,k_1\in K$. Denote $\vert H/D\vert$ by $n$ and let $kD$ be a generator of $K/D$. Since $(\vert H/D\vert,\vert K/D\vert)=1$, $k^{n}D$ is also a generator of $K/D$. Thus an element of $K $ is of the form $k^{ni}d$ for some $d\in D, i\in \mathbb Z$. For any $h\in H$,\begin{align*}
       k^{ni}d\wedge h &=\;^{k^{ni}}(d\wedge h)(k^{ni}\wedge h)\\
       &= (k^{ni}\wedge h)
      = (k\wedge h)^{ni}
       = k\wedge h^{ni}  &&\mbox{[Lemma}\ \ref{L:1.2.3}\ \textit{(vii)}] \\
       &= k\wedge d_1=1 , \text{ for some $d_1\in D.$}
   \end{align*}
   Any $g\in G$ is of the form $k_1h$ for some $k_1\in K$ and $h\in H$. Therefore, $k^{ni}d\wedge g =k^{ni}d\wedge k_1h=1$, and hence $K\leqslant Z^{\wedge}(G)$.
\end{proof}
\begin{corollary}
Let $K$ be a cyclic group, and $G=H\times K$. If order of $H$ is co-prime to the order of $K$, then $M(G)\cong M(H).$
\end{corollary}

The next lemma shows the universal property of the quotient group $G/{Z^{\wedge}(G)}$. It says that a surjective group morphism from any group $G$ to a capable group factors through $G/{Z^{\wedge}(G)}$.

\begin{lemma}\label{P:5:1}
Let $G$ be a group and $N$ be a normal subgroup of $G$. If $G/N$ is capable, then $Z^{\wedge}(G)\leqslant N.$
\end{lemma}
\begin{proof}Let $\pi:G\to G/N$ be the projection map. We will show that if $g\in G\setminus N$, then $g \notin Z^{\wedge}(G)$. Let $g\in G\setminus N$,  we have $\pi(g)\neq 1 $. Since $G/N$ is capable, $\pi(g)\notin Z^{\wedge}(G/N)$. Thus, there exists $\bar{h}\in G/N$ such that $\pi(g)\wedge \bar{h}\neq 1.$ Let $h\in G $ with $\pi(h)= \bar{h}$, and note that $g\wedge h \neq 1$. So $g\notin Z^{\wedge}(G)$.
\end{proof}

\begin{corollary}\label{C:5.3}
For a finite non-cyclic $p$-group $G$, we have $Z^{\wedge}(G)\leqslant \Phi(G).$
\end{corollary}
\begin{proof}
$G/\Phi(G)$ is a non-cyclic elementary abelian group. Thus $G/\Phi(G)$ is capable by \cite{RB1938} and the result follows by lemma \ref{P:5:1}.
\end{proof}
\begin{corollary}
Let $G$ be a finite non-cyclic group. If $G/[G,G] \cong \mathbb Z_{n_1}\times\cdots\times\mathbb Z_{n_k}\times \mathbb Z_{n_{k+1}}$, $n_i\vert n_{i+1}$ for $1\leq i\leq k$, then $Z^{\wedge}(G)\leqslant [G,G]G^{n_k}.$ In particular, if $G/[G,G]$ is capable, then $Z^{\wedge}(G)\leqslant [G,G]$
\end{corollary}
\begin{proof}
Observe that $G/[G,G]G^{n_k}\cong \mathbb Z_{n_1}\times\cdots\times\mathbb Z_{n_k}\times \mathbb Z_{n_k}$. Thus $G/[G,G]G^{n_k}$ is capable by \cite{RB1938} and the result follows by lemma \ref{P:5:1}.
\end{proof}
Recall that a unicentral group is a group whose exterior center coincides with its center. Denoting the minimal number of generators of $G$ by $d(G)$, we have the following theorem:
\begin{theorem}
Let $G$ and $H$ be finite $p$-groups. If $G/Z^{\wedge}(G)\cong H/Z^{\wedge}(H)$, then $d(G)= d(H)$. In particular, if $G$ and $H$ are exterior isoclinic, then $d(G)= d(H)$.
\end{theorem}
\begin{proof}
Suppose $d(G)=1$, then $G= Z^{\wedge}(G).$ Since $G/Z^{\wedge}(G)\cong H/Z^{\wedge}(H)$, we have $H= Z^{\wedge}(H).$ Thus $H$ is a unicentral group. We claim that a unicentral finite abelian group is cyclic. Suppose not, then $H\cong \mathbb Z_{n_1}\times\cdots\times\mathbb Z_{n_k}\times \mathbb Z_{n_{k+1}}$, and $H/H^{n_k}\cong \mathbb Z_{n_1}\times\cdots\times\mathbb Z_{n_k}\times \mathbb Z_{n_k}$, and hence $H/H^{n_k}$ is capable. By Proposition \ref{P:5:1}, we obtain that $Z^{\wedge}(H)\leqslant H^{n_k}$, a contradiction. Thus $H$ is cyclic and so $d(H)=1$. We may now assume that $d(G)\geq 2 $ and $d(H)\geq 2$. By corollary \ref{C:5.3}, $Z^{\wedge}(G)\leqslant \Phi(G)$ and so $\Phi(G/ Z^{\wedge}(G))= \Phi(G)/Z^{\wedge}(G)$. Thus, we have $\frac{G/ Z^{\wedge}(G)}{\Phi(G/ Z^{\wedge}(G))}= \frac{G/ Z^{\wedge}(G)}{\Phi(G)/Z^{\wedge}(G)}\cong G/\Phi(G).$ Hence, $d(G)= dim_{\mathbb F_p}(G/\Phi(G))= dim_{\mathbb F_p}(\frac{G/ Z^{\wedge}(G)}{\Phi(G/ Z^{\wedge}(G))}) = d(G/ Z^{\wedge}(G))$. Similarly $d(H)= d(H/ Z^{\wedge}(H))$. Moreover the isomorphism $G/Z^{\wedge}(G)\cong H/Z^{\wedge}(H)$ yields $d(G/Z^{\wedge}(G))= d(H/Z^{\wedge}(H))$, and hence $d(G)=d(H)$.
\end{proof}

\section{Invariance of q-Bogomolov Multiplier}
 In this section, we consider two distinct quotients of the $q$-Schur Multiplier of groups. Two notions of isoclinism are defined and it is observed that the considered quotients are invariant under the appropriate notion of isoclinism. 
The $q$-Bogomolov Multiplier, $B_0^q(G)$, and $\widehat{q}$-Bogomolov Multiplier , $\widehat{B_0^q (G)}$, are defined using the normal subgroups $M_0^q(G)$ and $\widehat{M_0^q(G)}$ of the $q$-Schur Multiplier described below.
\begin{align*}
M_0^q(G) &:= \left\langle x \wedge y | [x,y] = 1 \right\rangle.\\
B_0^q(G) &:= M^q(G)/M_0^q(G).\\
\widehat{M_0^q(G)} &:= \left\langle x \wedge y, \{(g,g)\} | [x,y] = 1, g^q = 1  \right\rangle.\\
\widehat{B_0^q(G)} &:= M^q(G)/\widehat{M_0^q(G)}.
\end{align*}
 We have following normal subgroups of $G$. \begin{align*}
     Z_q(G) &:= \{g\in Z(G)\mid \{(g,g)\}\in M_0^q(G)\}.\\ 
\widehat{Z_q(G)} &:= \{g\in Z(G)\mid g^q=1\}.
 \end{align*}

\begin{definition}
Two groups $G$ and $H$ are $q$-isoclinic if there exists isomorphisms $\alpha : G/Z_q(G) \rightarrow H/Z_q(H)$ and $\beta : G^q[G,G] \rightarrow H^q[H,H]$ with the property that whenever $\alpha(a_1Z_q(G)) = a_2Z_q(H)$ and $\alpha(b_1Z_q(G)) = b_2Z_q(H),$ then $\beta([a_1,b_1]) = [a_2,b_2]$ and  $\beta(a_1 ^q) = a_2 ^q$ for $a_1,b_1 \in G$, for $a_1,b_1 \in G$.
\end{definition}

\begin{definition}
Two groups $G$ and $H$ are $\widehat{q}$-isoclinic if there exists isomorphisms $\alpha : G/\widehat{Z_q(G)} \rightarrow H/\widehat{Z_q(H)}$ and $\beta : G^q[G,G] \rightarrow H^q[H,H]$ with the property that whenever $\alpha(a_1\widehat{Z_q(G)}) = a_2\widehat{Z_q(H)}$ and $\alpha(b_1\widehat{Z_q(G)}) = b_2\widehat{Z_q(H)},$ then $\beta([a_1,b_1]) = [a_2,b_2]$ and $\beta(a_1 ^q) = a_2 ^q$ for $a_1,b_1 \in G$.
\end{definition}

\begin{theorem}\label{T:4.2}
 $B_0^q(G)$ and $\widehat{B_0^q(G)}$ are invariant under $q$-isoclinism and $\widehat{q}$-isoclinism respectively. 
\end{theorem}

\begin{proof}
Given isomorphisms $\alpha$ and $\beta$  as in the definition of $q$-isoclinism, we first show that the following map is a well defined isomorphism:
\begin{align*}
\phi:\hspace{1.5mm} &G\wedge ^q G/M_0^q(G) \rightarrow H\wedge ^q H/M_0^q(H)\\
      & \overline{x_1 \wedge y_1} \mapsto \overline{x_2 \wedge y_2},\ 
       \overline{\{(x_1,x_1)\}} \mapsto  \overline{\{(x_2,x_2)\}},
\end{align*}
where $\alpha(x_1Z_q(G)) = x_2Z_q(H)$ and $\alpha(y_1Z_q(G)) = y_2Z_q(H)$. Towards that, consider the map $\phi : G\wedge ^q G \rightarrow\hspace{1mm} H\wedge ^q H/M_0^q(H)\hspace{1mm}$ with $x_1 \wedge y_1 \mapsto \overline{x_2 \wedge y_2}$ and  $\{(x_1,x_1)\} \mapsto  \overline{\{(x_2,x_2)\}}$, where $x_2$ and $y_2$ are obtained as above.
For $\alpha(\overline{x_1}) = \overline{x_2 a_2}$ and $\alpha(\overline{y_1}) = \overline{y_2 b_2}$, where $a_2,b_2 \in Z_q(H)$ we have,
$\overline{x_2a_2 \wedge y_2b_2} = \overline{x_2\wedge y_2}.$
Also,$\overline{\{(x_2 a_2, x_2 a_2)\}} = \overline{\{(x_2,x_2)\}\prod_{i=1}^{q-1}(x_2^{-1} \wedge {a_2} ^i)\{(a_2,a_2)\}}
                                  = \overline{\{(x_2,x_2)\}}.$ Suppose $[g_1,g_2] = 1$. For $h_1,h_2 \in H$ such that $\alpha(\overline{g_1}) = \overline{h_1}$ and $\alpha(\overline{g_2}) = \overline{h_2}$, we have $ [h_1,h_2] = \beta([g_1,g_2]) = 1$, which yields 
$\phi(\overline{g_1 \wedge g_2}) = \overline{h_1\wedge h_2}
                                = 1.$\\
Thus $\phi(M_0^q(G)) = 1$, and we obtain a well defined map $\phi :\hspace{1.5mm} G\wedge ^q G/M_0^q(G)\rightarrow H\wedge ^q H/M_0^q(H)$. To show that $\phi$ is a homomorphism, it suffices to verify that $\phi$ satisfies all the relations in the q-tensor square. We demonstrate a few below and the others follow similarly. For each $g_i \in G$, let $h_i \in H$ with $\alpha(\overline{g_i}) = \overline{h_i}$, $i \in \{1,2,3\}$, and suppose $\alpha(\overline{g_1g_2}) = \overline{h}$, then we have $h = h_1h_2a$ for some $a \in Z_q(H)$. Now,
\begin{align*}
\phi(\overline{g_1 g_2\wedge g_3}) &= \overline{h \wedge h_3}\\
                                  &= \overline{h_1h_2a \wedge h_3}\\
                                  &= \overline{h_1h_2 \wedge h_3}\\
                                  &= \overline{^{h_1}(h_2\wedge h_3)(h_1 \wedge h_3)}\\
                                  &= \phi(\overline{^{g_1}(g_2 \wedge g_3)(g_1\wedge g_3)}).
\end{align*}
 Let $g,g_1 \in G$ such that $\alpha(\overline{g}) = \overline{h}$, $\alpha(\overline{g_1}) = \overline{h_1}$ and $\alpha(\overline{gg_1}) = \overline{h'}$. Then we have $h' = hh_1a$ for some $a \in Z_q(H)$.

\begin{align*}
\phi(\overline{\{g_1g_2,g_1g_2\}}) &= \overline{\{(h,h)\}}\\
                               &= \overline{\{h_1h_2a,h_1h_2a\}}\\
                               &= \overline{\{h_1h_2,h_1h_2\}}\\
                               &= \overline{\{(h_1,h_1)\}\prod_{i=1}^{q-1}(h_1^{-1} \wedge (^{h_1^{1-q+i}}h_2)^i)\{(h_2,h_2)\}}\\
                               &= \phi(\overline{\{(g_1,g_1)\}\prod_{i=1}^{q-1}(g_1^{-1} \wedge (^{g_1^{1-q+i}}g_2)^i)\{(g_2,g_2)\}}).
\end{align*}
The inverse map $\psi:H\wedge ^q H/M_0^q(H) \rightarrow G\wedge ^q G/M_0^q(G)$ can be define similarly using $\alpha^{-1} $ and $\beta^{-1}$. Hence $\phi$ is an isomorphism and we have the following commutative diagram:
\[
 \begin{tikzcd}
1\arrow{r}& B^q_0(G)\arrow{r}\arrow{d}& G\wedge^q G/M_0^q(G) \arrow{r}\arrow{d}{\phi} & G^q{[G,G]} \arrow{r}\arrow{d}{\beta} &1\\
1\arrow{r}& B_0^q(H)\arrow{r} &H\wedge^q H/M_0^q(H) \arrow{r} & H^q{[H,H]}\arrow{r} &1.
\end{tikzcd}\\
\]
The maps $\phi$ and $\beta$ being isomorphism, yields $B_0^q(G)\cong B_0^q(H).$ The invariance of $\widehat{B_0^q(G)}$ under $\widehat{q}$-isoclinism can be proved similarly.
\end{proof}

\begin{Remark}
Let $G$ and $H$ be two groups, and $A \subset Z_q(G)$ and $B \subset Z_q(H)$. If there exists a pair of isomorphisms $\alpha : G/A \rightarrow H/B$ and $\beta : G^q[G,G] \rightarrow H^q[H,H]$ with the property that whenever $\alpha(a_1A) = a_2B$ and $\alpha(b_1A) = b_2B,$ then $\beta([a_1,b_1]) = [a_2,b_2]$ and  $\beta(a_1 ^q) = a_2 ^q$ for $a_1,b_1 \in G$, for $a_1,b_1 \in G$. Then using the arguments in the above theorem, we can prove that $B_0^q(G)\cong B_0^q(H)$.
\end{Remark}

\begin{Remark}
Let $G$ and $H$ be two groups, and $A \subset \widehat{Z_q(G)}$ and $B \subset \widehat{Z_q(H)}$. If there exists a pair of isomorphisms $\alpha : G/A \rightarrow H/B$ and $\beta : G^q[G,G] \rightarrow H^q[H,H]$ with the property that whenever $\alpha(a_1A) = a_2B$ and $\alpha(b_1A) = b_2B,$ then $\beta([a_1,b_1]) = [a_2,b_2]$ and  $\beta(a_1 ^q) = a_2 ^q$ for $a_1,b_1 \in G$, for $a_1,b_1 \in G$. Then using the arguments in the above theorem, we can prove that $\widehat{B_0^q(G)}\cong \widehat{B_0^q(H)}$.
\end{Remark}

\begin{Remark}
Note that $E_q^{\wedge}(G)\leqslant Z_q(G)$ and $E_q^{\wedge}(G)\leqslant \widehat{Z_q(G)}$. So if $G$ and $H$ are $q$-exterior isoclinic, then $B_0^q(G)\cong B_0^q(H)$ and $\widehat{B_0^q(G)}\cong \widehat{B_0^q(H)}$.
\end{Remark}
 
\section{Some Structural Results on q-Nonabelian Tensor Square of a Group}

In this section, we generalize some structural results proved in \cite{ADST} and \cite{RR2017}. We begin with the following lemma on the $q$-tensor square of abelian groups.

\begin{lemma}\label{L:5.1}
If $A$ is an abelian group, then the following statements hold:
\begin{enumerate}
    \item[(i)] $A\otimes^q A$ is abelian.
    \item[(ii)] If $A$ is a finitely generated group and $q$ be any non-negative integer, then $A\otimes^q A$ is a finitely generated group.
\end{enumerate}
\end{lemma}
\begin{proof}\textit{(i)} It suffices to show that the generators of $A\otimes^q A$ commute with each other.
By lemma \ref{L:1.2.3} $(iv)$, we have $[g \otimes h, g_1 \otimes h_1] =\ [g,h]\otimes [g_1,h_1]=1$ for all $g,g_1,h,h_1, \in G$. Using \eqref{eq:E:2.15}, we obtain $[\{(g,g)\},\{(h,h)\}] =\ g^q \otimes h^q$. Now applying
lemma \ref{L:1.2.3} $(vii)$ and then \eqref{eq:E:2.16} yields, $(g^q \otimes h^q) = (g^q \otimes h)^q = \{([g^q,h],[g^q,h])\}  = 1$. Furthermore, by \eqref{eq:E:2.13} we have
$\{(g,g)\}(g_1\otimes h_1)\{(g,g)\}^{-1} =\ g_1 \otimes h_1$, and hence the proof.
\par \textit{(ii)} Let $\{x_i, 1\leq i\leq n\}$ be the generators of $A$ with an ordering. We claim that $\{ x_i\otimes x_j,\{(x_i,x_i)\}|1\leq i,j\leq n\}$ is a generating set for $A\otimes^q A$.
Consider $a,b \in A$, where $a = \prod\limits_{i=1}^{n}x_i^{m_i}$ and $b = \prod\limits_{i=1}^{n}x_i^{l_i}$ for some $m_i,l_i\in \mathbb{Z}$. Since $A\otimes^q A$ is abelian, we have
$a \otimes b
=\ \prod\limits_{i=1}^{n}(x_i \otimes x_i)^{m_il_i}\prod\limits_{1\leq i<j\leq n}(x_i\otimes x_j)^{m_il_j}(x_j\otimes x_i)^{m_j l_i}$. Applying lemma \ref{L:1.2.3} $(ix)$, we obtain $\{(a,a)\} =\ \prod\limits_{i=1}^{n}\{(x_i,x_i)\}^{m_i}$ for $q$ odd,
and  when $q$ is even, using \eqref{eq:E:2.14} and expanding as before yields the proof.
\end{proof}

Now we come to the main theorem of this section.

\begin{theorem}\label{T:5.3}
Let $k$ be an odd integer. If $(g \otimes g)^k = 1$ for all $g \in G$, then $G\otimes^q G \cong \nabla^q(G) \times G \wedge^q G$.
\end{theorem}
\begin{proof}
Let $k = 2n+1$, for some $n \in \mathbb{Z}$.
Now consider the short exact sequence
$1 \rightarrow \nabla^q(G) \xrightarrow{\alpha}G\otimes^q G \rightarrow G\wedge^q G \rightarrow 1,$
where $\alpha$ is the inclusion map. Our aim is to construct a left splitting map for the above exact sequence which would yield the required direct product.
Towards that,
define $\alpha' : G\otimes ^q G \rightarrow \nabla^q(G)$ on the generators as given below.
\begin{align*}
\alpha'(g \otimes h) =& ((g\otimes h)(h \otimes g))^{-n}\\
\alpha'(\{(g,g)\}) =& (g \otimes g)^{n{{q}\choose{2}}},
\end{align*}for all $g,h \in G$. 
Note that, for all $g \in G$, we have
$\alpha' \circ \alpha (g\otimes g) 
= ((g \otimes g)(g \otimes g))^{-n}
= (g \otimes g)$.
Therefore it suffices to prove that $\alpha'$ is a well-defined homomorphism. The verification of the first two relations follows as in Theorem 4.13 of \cite{ADST}. 
Note that from Lemma \ref{L:1.2.3} $(i)$, we have $\nabla^q(G)$ is abelian. Now since $G$ acts trivially on $\Delta^q(G)$,
\begin{align*}
\alpha'(\{(g,g)\}(g_1 \otimes h_1)\{(g,g)\}^{-1}) =&\ (g \otimes g)^{n{{q}\choose{2}}}\big((g_1 \otimes h_1)(h_1\otimes g_1)\big)^{-n}\big((g \otimes g)^{n{{q}\choose{2}}}\big)^{-1}\\
=&\ \big(^{g^q}\big((g_1 \otimes h_1)(h_1\otimes g_1)\big)\big)^{-n}\\
=&\ \alpha'(^{g^q}g_1 \otimes ^{g^q}h_1).
\end{align*}
Similarly,
\begin{align*}
&\alpha'\Big(\{(g,g)\}\prod_{i=1}^{q-1}(g^{-1} \otimes (^{g^{1-q+i}}g_1)^i)\{(g_1,g_1)\}\Big)\\
&=\ (g \otimes g)^{n{{q}\choose{2}}}\prod_{i=1}^{q-1} {^{g^{1-q+i}}}\big((g^{-1}\otimes g_1^i)(g_1^i \otimes g^{-1})\big)^{-n}(g_1 \otimes g_1)^{n{{q}\choose{2}}}\\
&=\ (g \otimes g)^{n{{q}\choose{2}}}\big((g^{-1}\otimes g_1)(g_1 \otimes g^{-1})\big)^{(-n){{q}\choose{2}}}(g_1 \otimes g_1)^{n{{q}\choose{2}}} &&  \mbox{[Lemma\ }\ref{L:1.2.3}\ \textit{(x)}]\\
&=\ (g \otimes g)^{n{{q}\choose{2}}}\big((g_1\otimes g)(g \otimes g_1)\big)^{n{{q}\choose{2}}}(g_1 \otimes g_1)^{n{{q}\choose{2}}}\ &&\mbox{[Lemma}\ \ref{L:1.2.3}\ \textit{(v)} ]\\
&=\ \alpha'(\{(gg_1,gg_1)\})\ &&\mbox{[Lemma}\ \ref{L:1.2.3}\ \textit{(vi)}].
\end{align*}
The other relations can easily be verified, and the proof follows.
\end{proof}

The following Corollary generalizes \cite[Proposition 2.12]{RR2017}.

\begin{corollary}\label{C:5.4}
Let $q$ be an odd integer, then $G \otimes^q G \cong \nabla^q G \times G\wedge^q G$.
\end{corollary}
\begin{proof}
Using Theorem \ref{T:5.3} and lemma \ref{L:1.2.3} \textit{(ii)} yields the result.
\end{proof}

\begin{prop}\label{T:6.3}
If $\nabla^q G$ has unique $n^{th}$ roots where $n$ is even, then $G \otimes^q G \cong \nabla^q G \times G\wedge^q G$.
\end{prop}
\begin{proof}
Let $n = 2k$, $ k \in \mathbb{N}$.
Consider the short exact sequence $1 \to \nabla^q G \xrightarrow{\beta} G \otimes^q G \to G \wedge^q G \to 1$, where $\beta$ is the natural inclusion map. Define $\beta' :G\otimes^q G\rightarrow \nabla^q G$ by 
\begin{align*}
\beta'(g\otimes h) =& [(g\otimes h)(h\otimes g)]^{\frac{k}{n}},\\
\beta'(\{(g,g)\}) =& (g\otimes g)^{{\frac{-k}{n}}{{q}\choose{2}}}.
\end{align*}

 Proceeding as in theorem \ref{T:5.3} yields $\beta'$ is a well defined homomorphism. Moreover, $\beta'\beta(g\otimes g) = [(g\otimes g)(g\otimes g)]^{\frac{k}{n}}= (g\otimes g)$, and hence the proof .
 
\end{proof}

\section*{Acknowledgement} V. Z. Thomas acknowledges research support from SERB, DST, Government of India grant MTR/2020/000483.

\bibliographystyle{amsplain}
\bibliography{Bibliography}
\end{document}